\documentclass[12pt,a4paper,twoside]{article}

\usepackage[latin1]{inputenc}
\usepackage[brazil,english]{babel}
\usepackage{amsfonts}
\usepackage{amssymb}
\usepackage{amsmath}
\usepackage{amsthm}
\usepackage[usenames,dvipsnames]{color}
\usepackage{enumitem}
\usepackage[normalem]{ulem}
\usepackage{setspace}

\numberwithin{equation}{section}
\theoremstyle{plain} 
\newtheorem{theorem}{Theorem}[section]
\newtheorem*{theorem*}{Theorem}

\newtheorem{lemma}[theorem]{Lemma}
\newtheorem*{lemma*}{Lemma}
\newtheorem{prop}[theorem]{Proposition}

\theoremstyle{definition} 

\theoremstyle{remark} 
\newtheorem{remark}[theorem]{Remark}

\makeatletter
\newcommand{\aslabel}[1]{#1\def\@currentlabel{#1}}
\newcommand{\nowlabel}[1]{\def\@currentlabel{#1}}
\makeatother

\newcommand{\R}{{\mathbb{R}}}
\newcommand{\N}{{\mathbb{N}}}
\newcommand{\C}{{\mathbb{C}}}
\newcommand{\Z}{{\mathbb{Z}}}

\newcommand{\tref}[1]{\ref{#1}}
\newcommand{\pref}[1]{(\ref{#1})}

\newcommand{\pt}[1]{\left({#1}\right)}
\newcommand{\pq}[1]{\left[{#1}\right]}
\newcommand{\pg}[1]{\left\{{#1}\right\}}

\newcommand{\m}[1]{\left|{#1}\right|}

\newcommand{\hsp}{\hspace*{.5cm}}

\usepackage{textcomp,amssymb,wasysym}

\definecolor{collnk}{rgb}{.2,0.2,.6}
\definecolor{colcit}{rgb}{.1,0.5,.2}
\usepackage[colorlinks=true]{hyperref}\hypersetup{urlcolor=blue, citecolor=colcit,linkcolor=collnk}

\newcommand{\D}{\mathbb{D}}

\newcommand{\DD}{ \varXi} 
\newcommand{\esfq}{\Omega_{2q}}
\newcommand{\esfp}{\Omega_{2p}}
\newcommand{\esfi}{\Omega_\infty}
\newcommand{\prode}{\Omega_{2q}\times\Omega_{2p}}
\newcommand{\Rmnq}{R_{m,n}^{q-2}}

\begin{document}
\title{Characterization of Strict Positive Definiteness on products of complex spheres}
\author{Mario H. Castro$^a$, Eugenio Massa$^b$, and Ana Paula Peron$^b$.
\\\scriptsize{$^a$Departamento de Matem\'{a}tica, UFU  - Uberlândia (MG), Brazil. }\\ \scriptsize{
$^b$Departamento de Matem\'{a}tica, ICMC-USP  - S\~{a}o Carlos, Caixa Postal 668, 13560-970  S\~{a}o Carlos (SP), Brazil.}\\ \scriptsize{
		e-mails: mariocastro@ufu.br,\,\,\,\,  eug.massa@gmail.com,\,\,\,\, apperon@icmc.usp.br.
}
}

\maketitle 


\begin{abstract}
In this paper we consider Positive Definite functions on products  $\prode$ of complex spheres, and we obtain a condition, in terms of the coefficients in their disc polynomial expansions, which is necessary and sufficient for the function to be Strictly Positive Definite.  The result includes also the more delicate cases in which $p$ and/or $q$ can be $1$ or $\infty$.

The condition we obtain states that a suitable set in $\Z^2$, containing the indexes of the strictly positive coefficients in the expansion, must intersect every product of arithmetic progressions.
%
\\
\noindent{\bf MSC:} 42A82; 42C10.
\\
\noindent{\bf Keywords:} Strictly Positive Definite Functions, Product of Complex Spheres,  Generalized Zernike Polynomial.
\end{abstract}

\section{Introduction}
The main purpose of this paper is to obtain a characterization of Strictly Positive Definite functions on products of complex spheres, in terms of the coefficients in their disc polynomial expansions: these results are contained in  the Theorems \ref{th_main}, \ref{th_main_1p} and \ref{th_main_11}.

Positive Definiteness and Strict Positive Definiteness are important in many applications, for example, Strict Positive Definiteness is required in certain interpolation problems in order to guarantee the unicity of their solution. 
From a theoretical point of view, the problem of characterizing both Positive Definiteness and Strict Positive Definiteness has been considered in many recent papers, in different contexts. More details on the applications and the literature related to this problem will be given in Section \tref{sec_liter}.  

\par\medskip

Let $\varOmega$ be a nonempty set. A kernel $K: \varOmega \times \varOmega \to \C$ is called {\em Positive Definite} (PD in the following) on $\varOmega$ when 
\begin{equation}\label{eq-quad-form-geral}
\sum_{\mu,\nu=1}^L c_\mu\overline{c_\nu} K(x_\mu,x_\nu) \geq 0, 
\end{equation}
for any $L\ge1$, $c=(c_1,\ldots,c_L) \in \C^L$ and any subset  $X:=\{x_1,\ldots,x_L\} $ of distinct points in $\varOmega$.   Moreover, $K$ is {\em Strictly Positive Definite} (SPD in the following) when it is Positive Definite and the inequality above is strict for  $c\neq0$.

   If $S^q$ is the $q$-dimensional unit sphere in the Euclidean space $\R^{q+1}$,  we say that a continuous function $f:[-1,1] \to \R$ is PD (resp. SPD) on $S^q$,  when the associated kernel $K(v,v'):=f(v\cdot_\R v')$ is PD (resp. SPD) on $S^q$ (here ``~$\cdot_\R$~"  is the usual inner product in $\R^{q+1}$). 
   In \cite{scho-42} it  was proved that 
     a continuous function  $f$ is PD on $S^q$, $q\geq1$,  if, and only if, it admits an expansion in the form
   \begin{equation}\label{eq-scho}\begin{array}{ccc}
   &\displaystyle f(t)=\sum_{m\in\Z_+} a_mP_m^{(q-1)/2}(t),\quad t\in[-1,1],&\\&\mbox{where $\sum a_mP_m^{(q-1)/2}(1)<\infty$ and $a_m\geq0$ for all $m\in\Z_+$. }\end{array}
   \end{equation}
   In  \pref{eq-scho}, $P_m^{(q-1)/2}$ are the Gegenbauer polynomials of degree $m$ associated to $(q-1)/2$ (see \cite[page 80]{szego}) and $\Z_+=\N\cup\pg{0}$. 
   In
   \cite{debao-sun-valdir} it was proved that the function $f$ in  \pref{eq-scho} is also SPD on $S^q$, $q\geq2$ if, and only if, the set 
   $
   \{m\in\Z_+: a_{m}>0\}
   $
   contains an infinite number of odd and of even numbers. This condition is equivalent to asking that 
   \begin{equation}\label{eq_inters_Sd}
   \{m\in\Z_+: a_{m}>0\}\cap (2\N+x)\neq \emptyset \qquad\mbox{for every $x\in\N$}.
       \end{equation}

    The complex case is defined in a similar way:  if $\esfq$ is the unit sphere in $\C^q$, $q\geq2,$ and $\D$ is the unit closed disc in $\C$, 
    then  a continuous function $f:\D \to \C$ is said to be PD (resp. SPD) on $\esfq$ if the associated kernel $K(z,z'):=f(z\cdot z')$ is PD (resp. SPD) on $\esfq$, where  ``~$\cdot$~"  is the  usual inner product in $\C^q$.  
    As proved in  \cite{P-valdir-pd-esfcompl},  a continuous function $f:\D \to \C$  is PD on $\esfq$, $q\geq2$ if, and only if, it has the representation in series of the form 
    \begin{equation}\label{eq-pd-esfq}\begin{array}{ccc}
     &\displaystyle f(\xi)=\sum_{m,n\in\Z_+} a_{m,n}R_{m,n}^{q-2}(\xi),\quad \xi\in{\D},&\\
    &\mbox{where $\sum a_{m,n}<\infty$ and $a_{m,n}\geq0$ for all $m,n\in\Z_+$.}&
\end{array}
     \end{equation}
    The functions $R_{m,n}^{q-2}$ in \pref{eq-pd-esfq} are the  {\em disc polynomials}, or {\em generalized Zernike polynomials} (see Equation \pref{eq-def-pol-disc}). 
     The condition for $f$ to be SPD was obtained in  \cite{meneg-jean-traira,P-valdir-complexapproach}: 
    $f$ as in \pref{eq-pd-esfq}  is SPD on $\esfq$ if, and only if,
    the set
    $
    \{m-n\in\Z: a_{m,n}>0\}
    $
    intersects every full arithmetic progression in $\Z$, that is, 
      \begin{equation}\label{eq_inters_Oq}
     \{m-n\in\Z: a_{m,n}>0\}\cap (N\Z+x) \neq \emptyset
     \qquad\mbox{for every $N,x\in\N$.}
     \end{equation}
    
    The characterization of SPD functions on the spheres $S^1$, $\Omega_2$, $S^\infty$  and $\Omega_\infty$ were also considered in  \cite{P-valdir-claudemir-spd-compl,men-spd-analysis,meneg-jean-traira}, obtaining similar results (see also in Section \ref{sec_S1}).

    Products of real spheres where considered in \cite{P-jean-men-pdSMxSm, jean-men, P-jean-menS1xSm,	P-jean-menS1xS1}:
    a continuous PD function on $S^q\times S^p$, $q,p\geq1$ can be written as 
    \begin{equation}\label{eq-pd-esfSd}\begin{array}{c}
   \displaystyle  f(t,s)=\sum_{m,k\in\Z_+} a_{m,k}P_m^{(q-1)/2}(t)P_k^{(p-1)/2}(s),\quad t,s\in[-1,1],\\ 
   \mbox{where  $\sum a_{m,k}P_m^{(q-1)/2}(1)P_k^{(p-1)/2}(1)<\infty$ and $a_{m,k}\geq0$ for all $m,k\in\Z_+$,}
\end{array} 
   \end{equation}
   and, for $q,p\geq2$, it is also SPD on $S^q\times S^p$ if, and only if, the following condition, obtained in \cite{jean-men}, holds true:  in each intersection of the set 
      $
      \{(m,k)\in\Z_+^2: a_{m,k}>0\}
      $
     with the four sets $(2\Z_++x)\times(2\Z_++y),\ x,y\in\pg{0,1}$, there exists a sequence $(m_i,k_i)$ such that $m_i,k_i\to \infty$. 
     In fact, this condition is equivalent to the following one:
       \begin{equation}\label{eq_inters_SpSq}
    \{(m,k)\in\Z_+^2: a_{m,k}>0\}\cap (2\N+x)\times(2\N+y)\neq \emptyset
          \qquad\mbox{for every $x,y\in\N$.}
          \end{equation}

      Again, when considering $S^1$ in the place of $S^q$ and/or $S^p$, similar (but not analogous) results are obtained: see \cite{ P-jean-menS1xSm,	P-jean-menS1xS1}  and Section \ref{sec_S1}.

   \subsection{Main results}
The purpose of this paper is to consider the same kind of problems described above for the case of the   products $\prode$ of two complex spheres.  

The characterization of Positive Definiteness in this setting was obtained 
 in \cite[Theorem 7.1]{P-berg-porcu_gelfand} for $ q,p\in\N,\;q,p\geq2$: it was proved that  a continuous function $f:\D \times \D \to \C$  is PD on $\prode$   if, and only if, it admits an expansion in the form 
\begin{equation}\label{eq-pd-prod-esf-compl} 
\begin{array}{ccc}
&\displaystyle f(\xi,\eta)=\sum_{m,n,k,l\in\Z_+} a_{m,n,k,l}R_{m,n}^{q-2}(\xi)R_{k,l}^{p-2}(\eta),\quad (\xi,\eta)\in{\D}\times{\D},&\\
&\mbox{where $\sum a_{m,n,k,l}<\infty$ and $a_{m,n,k,l}\geq0$ for all $m,n,k,l\in\Z_+$.}&\end{array}
\end{equation}
If $p$ and/or $q$ can take the values $1$ or $\infty$, a characterization of Positive Definiteness is also known (see in Section \ref{sec_PD_on_prod}), except for the case $p=q=\infty$, which we address in Theorem \ref{th_PDinfty}. 
In fact, if we define $ R^{\infty}_{m,n}(\xi):={\xi\vphantom{\overline\xi}}^{m}\overline\xi^{n},\ \xi\in\D$, then the characterization \pref{eq-pd-prod-esf-compl} holds for $ q,p\in\N\cup\pg{\infty},\;q,p\geq2$.

Our main results are contained in the following theorems, where we characterize SPD functions on the product of two complex spheres $\prode$, $q,p\in\N\cup\pg{\infty}$, in terms of the coefficients in their expansions. 
\begin{theorem}\label{th_main}
 Let $ q,p\in\N\cup\pg{\infty},\ q,p\geq2$. A continuous function $f:\D \times \D \to \C$, which is PD   on $\prode$,  is also SPD  on $\prode$ if, and only if, considering its expansion as in \pref{eq-pd-prod-esf-compl}, the set $$J':=\pg{(m-n,k-l)\in\Z^2:\ a_{m,n,k,l}>0}$$ intersects every product of full arithmetic progressions in $\Z$, 
	that is, 
	  \begin{equation}\label{eq_inters_Opq}
	J'\cap (N\Z+x)\times (M\Z+y)\neq \emptyset\qquad \mbox{for every $N,M,x,y\in\N$.}
	\end{equation}
\end{theorem}

It is worth noting the similarities between the characterizations of SPD in the various cases described here, actually, they can  always be reduced to a condition on the intersection between a set constructed with the indexes of the nonnegative coefficients in the expansion  of the function, and certain arithmetic progressions or products of them: compare the conditions (\ref{eq_inters_Sd}-\ref{eq_inters_Oq}-\ref{eq_inters_SpSq}-\ref{eq_inters_Opq}).

\par\smallskip

When $p$ and/or $q$ can take the value $1$,
we obtain the  following characterizations.
 \begin{theorem}\label{th_main_1p}
 	Let $2\leq p\in\N\cup\pg{\infty}$. A continuous function $f:\partial\D \times \D \to \C$, which is PD   on $\Omega_2\times\esfp$,  is also SPD  on $\Omega_2\times\esfp$ if, and only if, considering its expansion as
 	\begin{equation}\label{eq-pd-prod-esf-complO2p}\begin{array}{rcl}
 	&\displaystyle f(\xi,\eta)=\sum_{m\in\Z,\,k,l\in\Z_+} a_{m,k,l}\xi^mR_{k,l}^{p-2}(\eta),\quad (\xi,\eta)\in\partial{\D}\times{\D},&\\
 	&\mbox{where $\sum a_{m,k,l}<\infty$ and $a_{m,k,l}\geq0$ for all $m\in\Z,\,k,l\in\Z_+$,}&
 	\end{array}
 	\end{equation} the set $$\pg{(m,k-l)\in\Z^2:\ a_{m,k,l}>0}$$ intersects every product of full arithmetic progressions in $\Z$.
 \end{theorem}

 \begin{theorem}\label{th_main_11}
 A continuous function $f:\partial \D \times \partial\D \to \C$, which is  PD   on $\Omega_2\times\Omega_2$,  is also SPD  on $\Omega_2\times\Omega_2$ if, and only if, considering its expansion as
 	\begin{equation}\label{eq-pd-prod-esf-complO2}\begin{array}{rcl}
 	&\displaystyle f(\xi,\eta)=\sum_{m,k\in\Z} a_{m,k}\xi^m\eta^k,\quad (\xi,\eta)\in\partial{\D}\times\partial{\D},&\\
 	&\mbox{where $\sum a_{m,k}<\infty$ and $a_{m,k}\geq0$ for all $m,k\in\Z$,}&\end{array}
 	\end{equation} the set $$\pg{(m,k)\in\Z^2:\ a_{m,k}>0}$$ intersects every product of full arithmetic progressions in $\Z$.
 \end{theorem}
 
We observe that the   Theorems  \ref{th_main_1p} and  \ref{th_main_11} 
 will follow immediately from the same proof as Theorem  \ref{th_main}, after  rewriting the expansions \pref{eq-pd-prod-esf-complO2p} and \pref{eq-pd-prod-esf-complO2} in order to be formally identical to \pref{eq-pd-prod-esf-compl} (see Lemma \ref{lm_charDD}). This is a remarkable fact considering that, in the real case, when the product involves the sphere $S^1$ (see \cite{P-jean-menS1xS1,P-jean-menS1xSm}) one had to use quite  different arguments with respect to  the higher dimensional case in \cite{jean-men}.
   We remark however that Theorem  \ref{th_main_11} is not new, as it is a particular case of the main result in \cite{men-gue-toro}.    
 
 \par\medskip
 
 This paper is organized in the following way. In  Section  \ref{sec_liter} we discuss some further literature related to our problem.  
   In Section \ref{sec_teoria} we set our notation and discuss some known results that will be used later.
  Theorem  \ref{th_main} is proved in Section \ref{sec_proofmain}. 
    In Section \ref{sec_infty} we state and prove the mentioned  characterization of PD functions on $\esfi\times\esfi$. Finally,  Section \ref{sec_S1} is devoted to  showing how one can deduce, from Theorem \ref{th_main_11},  the characterization of SPD functions on $S^1\times S^1$ proved in \cite{P-jean-menS1xS1}.

\subsection{Literature}\label{sec_liter}
Since the first results on  
 Positive Definite functions on real spheres,  obtained by Schoenberg in his seminal paper  (\cite{scho-42}), such functions were found to be relevant and have been studied in several distinct areas. In fact, they are both used by researchers directly interested in applied sciences, such as geostatistics, numerical analysis, approximation theory  (cf. \cite{cheney-approx-pd, CheneyLight-book, gneiting-2013, porcu-bev-gent}), and by theoretical researchers aiming at further generalizations that, along with their theoretical importance, could become useful in other practical problems. 
 
 One important motivation for characterizing  Strictly Positive Definite functions comes from  certain interpolation problems, where the interpolating function is generated by a Positive Definite kernel. Actually, the unicity of the solution of the interpolation problem is guaranteed only if the generating kernel is also Strictly Positive Definite (cf. \cite{light-cheney,cheney-xu}): consider, for instance, the interpolation function 
 $$
 F(x) = \sum_{j=1}^Lc_jK(x,x_j), \quad x\in \varOmega,
 $$
 where $X=\{x_1,\ldots,x_L\} \subset \varOmega$ is given and   $K$ is a known Strictly Positive Definite kernel in $\varOmega$; then the matrix of the system obtained from the interpolation conditions  $F(x_i)=\lambda_i$, $i=1,\ldots,L$, is the  matrix $[K(x_i,x_j)]$, whose determinant is positive, thus giving a unique solution for the system. 
In particular, the case where $\varOmega$  is a real sphere is very important in applications where one needs to 
assure unicity for interpolation problems with datas given on the Earth surface (which can be identified with the real sphere $S^2$).
 Also, the case where $\varOmega$ is the product of a sphere with some other set 
turns out to be  of particular interest for its application to geostatistical problems  in space and time, whose natural domain is $S^2\times \R$ (see \cite{porcu-bev-gent} and references therein).
Immediate applications in the case of complex spheres are less obvious: we  refer to \cite{P-massa-porcu-montee}, where parametric families of Positive Definite functions on complex spheres are provided. It is also worth noting that  the Zernike polynomials are used in applications such as   optics and optical engineering (cf. \cite{ramos-et-al-zernike-opt,torre} and references therein). 

Motivated by these and other  applications, several papers appeared dealing with the theoretical problem of characterizing  Positive Definiteness and Strict Positive Definiteness:
along with those already mentioned in the introduction, we cite \cite{musin-multi-pd}, where a characterization of  real-valued multivariate Positive Definite functions on $S^q$ is obtained, and \cite{ yaglom,hannan,men-rafaela}, where   matrix-valued Positive Definite functions are investigated. 

In \cite{porcu-berg},  the characterization in \cite{scho-42} is extended to the case of  Positive Definite  functions on the cartesian  product of $S^q$ times a locally compact group $G$, 
which includes the mentioned case  $S^q\times \R$
and also generalizes the  result obtained in \cite{P-jean-men-pdSMxSm} about Positive Definite functions on products of real spheres. 
Also,  the Positive Definite functions on Gelfand pairs and on products of them were characterized in \cite{P-berg-porcu_gelfand}, 
while those  on the product of a locally compact group with $\esfi$ in \cite{P-berg-porcu-Omega-inf}.

Concerning the characterization of Strictly Positive Definite functions, we cite also the cases of compact two-point homogeneous spaces and    products of them (\cite{barbosa-men, men-victor-prod-esf-esp_homg}) 
and the case of a torus (\cite{men-gue-toro}).

\section{Notation and known results}\label{sec_teoria}
 We first give a brief introduction on the  {disc polynomials} that appear in the Equations  \pref{eq-pd-esfq} and \pref{eq-pd-prod-esf-compl}:
 for $2\leq q\in\N$, the function $R_{m,n}^{q-2}$, defined in the disc $\D=\pg{\xi\in\C:|\xi|\leq 1}$, is called {\em disc polynomial} (or {\em generalized Zernike polynomial})  of degree $m$ in $\xi$ and $n$ in $\overline\xi$ associated to ${q-2}$, and can be written as (see \cite{koor-II})
\begin{equation}\label{eq-def-pol-disc}
R_{m,n}^{{q-2}}(\xi)=r^{|m-n|}\,e^{i(m-n)\phi}\,R_{\min\{m,n\}}^{({q-2},\,|m-n|)}(2r^2 -1), \quad \xi=re^{i\phi}\in\D,\ m,n\in\Z_+,\end{equation}
where $R_{k}^{(\alpha,\beta)}$ is the usual Jacobi polynomial of degree $k$ associated to the real numbers $\alpha,\beta>-1$ and normalized by $R_{k}^{(\alpha,\beta)}(1)=1$ (see \cite[page 58]{szego}).

For future use we also define 
\begin{eqnarray}\label{eq_Rm1inf}
R^{\infty}_{m,n}(\xi)=R^{-1}_{m,n}(\xi)&:=&{\xi\vphantom{\overline\xi}}^{m}\overline\xi^{n},\qquad \xi\in\D\,.
\end{eqnarray}

It is well known (see \cite{koor-II,koor-london}) 
 that the disc polynomials, as well as  those defined  in \pref{eq_Rm1inf},  satisfy, for $q\in\N\cup\{\infty\}$, $\xi\in\D$, and  $m,n\in\Z_+$,
\begin{eqnarray} \label{eq_modulo-Rmn}
&R_{m,n}^{q-2}(1) = 1, \quad 
|R^{q-2}_{m,n}(\xi)| \leq 1, &\quad 
\\\label{eq_propridd_Rmn}
&R_{m,n}^{q-2} (e^{i\phi}\xi)  = e^{i(m-n)\phi}R_{m,n}^{q-2} (\xi),&  \quad \phi\in\R, 
\\\label{eq_proprconj_Rmn}
&R^{q-2}_{m,n}\pt{\,\overline\xi\,}=\overline{ R^{q-2}_{m,n}(\xi)}.&
\end{eqnarray}

Observe that, by \pref{eq_modulo-Rmn}, the series in \pref{eq-pd-esfq} and \pref{eq-pd-prod-esf-compl} converge uniformly in their domain.  Moreover, the characterization in  \pref{eq-pd-prod-esf-compl} implies that 
the functions $(\xi,\eta)\mapsto R_{m,n}^{q-2}(\xi)R_{k,l}^{p-2}(\eta)$ are PD on $\prode$ for all $m,n,k,l\in\Z_+$ (and, by \pref{eq-pd-esfq}, the functions   $\xi\mapsto R_{m,n}^{q-2}(\xi)$ are PD on $\esfq$).

	Another important property is contained in the following lemma.
	\begin{lemma}\label{lm_Rto0}
		If $q\in\N\cup\pg{\infty}$ and $\xi\in\D'=\{\xi\in\C:|\xi|< 1\}$, then
		\begin{equation}\label{eq_Rto0}
		\lim_{\stackrel{m+n\to\infty}{m\neq n}}R_{m,n}^{q-2}(\xi)=0\,.
		\end{equation}
		
		If $q\in\N\cup\pg{\infty}$ and $\xi=e^{i\phi}\in\partial\D$ then
		\begin{equation}\label{eq_Reitet}
		R_{m,n}^{q-2}(e^{i\phi})=e^{i(m-n)\phi}\,.
		\end{equation}
			\end{lemma}
	For the proof of \pref{eq_Rto0} when $q\geq2$ see \cite{meneg-jean-traira}.
	It is worth noting that the limit is true even without the condition ${m\neq n}$, except for the special case $q=2$ and $\xi=0$.
	On the other hand, \pref{eq_Reitet} follows from  (\ref{eq_modulo-Rmn}-\ref{eq_propridd_Rmn}).\\

\subsection{Positive Definiteness on complex spheres}\label{sec_PD_on_sing}
As we anticipated in the introduction, it is known by \cite{P-valdir-pd-esfcompl} that a continuous function $f:\D\to\C$ is PD on $\esfq$, $2\leq q\in\N$, if, and only if, the coefficients $a_{m,n}$ in  the series representation \pref{eq-pd-esfq} satisfy $\sum a_{m,n}<\infty$ and $a_{m,n}\geq0$ for all $m,n\in\Z_+$. 

 In the case of  the complex sphere $\Omega_2$, when associating a continuous function $f$ to a kernel via the formula $K(z,z'):=f(z\cdot z')$, one has that  $z\cdot z'\in\partial\D$ for every $z,z'\in\Omega_2$, then  it becomes  natural to consider functions  $f$ defined in $\partial\D$.
The PD functions on $\Omega_2$ were also characterized in 
\cite{P-valdir-pd-esfcompl}, namely, $f:\partial\D\to \C$ is PD on $\Omega_2$ if, and only if,
\begin{equation}\label{eq-pd-prod1}\begin{array}{ccc}
&\displaystyle f(\xi)=\sum_{m\in\Z} a_{m}\xi^m,\quad \xi\in{\partial \D},&\\
&\mbox{where $\sum a_{m}<\infty$ and $a_{m}\geq0$ for all $m\in\Z$.}&\end{array}
\end{equation}
In order to write this formula as \pref{eq-pd-esfq}, and then to be able to use the same expansion for all $q\in\N$, we use the polynomials $R_{m,n}^{-1}$ defined in \pref{eq_Rm1inf} and we rearrange the coefficients in \pref{eq-pd-prod1} so that 
\begin{equation}\label{eq-pd-prod1mn}
f(\xi)
=\sum_{m,n\in\Z_+} a_{m,n}R_{m,n}^{-1}(\xi),\quad \xi\in\partial\D,
\end{equation}
with the additional requirement that  $a_{m,n}=0 $ if $mn>0$, implying that 
$$\begin{cases}
a_{m,0}:=a_m,&m\geq 0,\\ a_{0,m}:=a_{-m},&m\geq 0.
\end{cases}$$
In this way, $f$ is PD on $\Omega_2$ if, and only if, it satisfies the characterization \pref{eq-pd-esfq} with $a_{m,n}=0 $ for $mn>0$ and $\partial\D$ in the place of $\D$.

The complex sphere  $\esfi$ is defined as the sphere of the  sequences in the  Hilbert complex space $\ell^2(\C)$  having unitary norm.
In  \cite{chris-ressel-pd}, it was proved that  a continuous function $f:{\D}\to \C$ is PD on $\esfi$ if, and only if, it admits the series representation 
\begin{equation}\label{eq-pd-esfi}\begin{array}{ccc}
&\displaystyle f(\xi)=\sum_{m,n\in\Z_+} a_{m,n}{\xi\vphantom{\overline\xi}}^m\overline{\xi}^n,\quad \xi\in{\D},&\\
&\mbox{where $\sum a_{m,n}<\infty$ and $a_{m,n}\geq0$ for all $m,n\in\Z_+$,}&\end{array}
\end{equation}
 which becomes analogous to the characterization \pref{eq-pd-esfq} if we use the definition of $R_{m,n}^\infty$ in  \pref{eq_Rm1inf}.
It is also worth noting  that $f$  is PD on $\Omega_\infty$ if, and only if, $f$  is PD on $\Omega_{2q}$ for every $q\geq2$.

\subsection{Positive Definiteness on products of spheres}
\label{sec_PD_on_prod}
 From now on, in order to simplify the exposition,  we will use the symbol $\DD$ to designate either  $\partial \D$ or $\D$, depending if we are  considering, respectively,  the sphere $\Omega_2$ or a higher dimensional sphere.

 When considering products of spheres $\prode$, $ q,p\in\N\cup\pg\infty$, a continuous functions $f:\DD\times\DD \to \C$ is said to be PD (resp. SPD) on $\prode$, if the associated kernel  
 \begin{equation}\label{eq-Kfromfprod}
 K:[\prode]\times[\prode]\ni (\,(z,w),(z',w')\,)\mapsto f(z\cdot z',w\cdot w')
 \end{equation}
 is PD (resp. SPD)  on $\prode$.

In this section we will justify the following claim:
\begin{lemma}\label{lm_charDD} 
 A continuous function $f:\DD \times \DD \to \C$  is PD on $\prode$, $ q,p\in\N\cup\pg\infty$,   if and only if, it admits an expansion in the form 
\begin{equation}\label{eq-pd-prod-esf-compl_DD} 
\begin{array}{ccc}
&\displaystyle f(\xi,\eta)=\sum_{m,n,k,l\in\Z_+} a_{m,n,k,l}R_{m,n}^{q-2}(\xi)R_{k,l}^{p-2}(\eta),\quad (\xi,\eta)\in{\DD}\times{\DD},&\\
&\mbox{where $\sum a_{m,n,k,l}<\infty$ and $a_{m,n,k,l}\geq0$ for all $m,n,k,l\in\Z_+$,}&\end{array}
\end{equation}
  adding the requirement that $a_{m,n,k,l}=0$ if $q=1$ and $mn>0$ (resp. $p=1$ and $kl>0$).
\end{lemma}
Lemma \ref{lm_charDD} is a generalization of the characterization  \pref{eq-pd-prod-esf-compl}  to include  the cases when $q,p$ can take the values $1$ or $\infty$, replacing $\D$ with $\DD$  and redefining the coefficients in the series, where $p$ or $q$ is $1$, as we did in Equation \pref{eq-pd-prod1mn}.

In order to justify the claim, we will use results from \cite{P-berg-porcu_gelfand} and \cite{P-berg-porcu-Omega-inf}, which are stated in a more general setting. 
Let   $U(p)$ be the locally compact group of the unitary $p\times p$ complex matrices. A  continuous
function
$\widetilde \Phi:U(p)\to \C$
is called Positive Definite on $U(p)$
 if the kernel
$(A,B)\mapsto\widetilde\Phi(B^{-1}A)$ is 
  Positive Definite on $U(p)$ {(see \cite[page 87]{Berg})}.

The following remark will be useful to translate from this setting to the case of complex spheres in which we are interested 
(see also \cite[Section 6]{P-berg-porcu_gelfand}).
\begin{remark}\label{rem_U_Om}
Let $\Phi:\DD\to\C$ and $\widetilde \Phi:U(p)\to\C$ be related by $\widetilde\Phi(A)=\Phi(Ae_p\cdot e_p)$, where $e_p= (1,0,\ldots,0)\in \esfp$.

Then $\widetilde\Phi(A)$ depends only on the upper-left  element $[A]_{1,1}$ and it can be seen by the definition of Positive Definiteness that  $\widetilde\Phi$ is PD on $U(p)$ if, and only if, $\Phi$ is PD on $\esfp$.

 Moreover, $\widetilde\Phi$ is continuous if, and only if, $\Phi$ is, since $M:U(p)\to \DD:A\mapsto [A]_{1,1}$ is continuous and admits a continuous right 
inverse  $$M^-:\DD\to U(p):\xi\mapsto M^-(\xi)\ \text{ such that }\ [M^-(\xi)]_{1,1}=\xi\,.$$  \end{remark}

Now Lemma \ref{lm_charDD} is obtained as follows:
\begin{enumerate}
\item When $ q,p\in\N,\;q,p\geq2$,   the lemma is exactly the characterization \pref{eq-pd-prod-esf-compl}.
\item When $q=1$ and $p\in\N$ (or vice-versa) we can use  Corollary 3.5 in \cite{P-berg-porcu_gelfand}, observing that we can  identify functions on $\Omega_2$ with periodic functions on $\R$, and we can take the locally compact group $L=U(p)$,  obtaining a characterization for PD functions on $\Omega_2\times U(p)$. Then we can translate the characterization from $U(p)$ to $\esfp$, using Remark \ref{rem_U_Om}.

\item When $q=\infty$ and $p\in\N$  (or vice-versa) we can use  Theorem 1.3 in \cite{P-berg-porcu-Omega-inf}, taking the locally compact group $L=U(p)$ and proceeding as above.

\item When $q=p=\infty$  the claim is a consequence of  Theorem \ref{th_PDinfty} in Section \ref{sec_infty}.
\end{enumerate}

\section{Proof of the main results}\label{sec_proofmain}
In the following we  
will need to consider matrices  whose elements are described by many indexes: for this we will write  
$$
\pq{b_{i,j,k,l,...}}_{i=1,..,I,\; j=i,..,J,\; ...}^{k=1,..,K,\;l=1,..,L,\; ...}\,,
$$ 
where the indexes in the lower line are intended to be line indexes and those in the above line are column indexes. 
Also, we will  specify the indexes alone when their ranges are clear.

Let $q,p\in \N\cup\pg{\infty}$. From \pref{eq-quad-form-geral} and \pref{eq-Kfromfprod}, the definition of Positive Definiteness on $\esfq\times\esfp$, for a 
 continuous function  $f:\DD\times\DD\to \C$, takes the form 
\begin{equation}\label{eq-quad-form}
\sum_{\mu,\nu=1}^Lc_\mu\overline{c_\nu}f(z_\mu\cdot z_\nu, w_\mu\cdot w_\nu) \geq0
\end{equation}
for all $L\geq1$, $(c_1,c_2,\ldots,c_L)\in\C^L$ and $X=\{(z_1,w_1),(z_2,w_2),\ldots,(z_L,w_L)\}\subset\esfq\times\esfp$.
As a consequence, if we define the matrix  $A_X$ associated to the function $f$  and to the set $X$ by
\begin{equation}\label{eq-def-AX}
A_X:= [f(z_\mu\cdot z_\nu, w_\mu\cdot w_\nu)]^{\mu=1,\ldots, L}_{\nu=1,\ldots, L}\,,
\end{equation}
then: 
\begin{itemize}
\item 
$f$ is PD if, and only if, for every choice of $L$,  $X$, and $c^t = (c_1,c_2,\ldots,c_L)$,
$$
\overline c^t A_X c \geq 0,
$$
that is, $A_X$ is a Hermitian and positive semidefinite matrix (see \cite[page 430]{horn-joh-matrix});
\item $f$ is also SPD if, and only if, for every choice of $L$ and $X$,
$$
\overline c^t A_X c = 0 \Longleftrightarrow c=0,
$$
that is, $A_X$ is a positive definite matrix.
\end{itemize}

 Let now $f$ be a continuous function, PD  on $\esfq\times\esfp$, which we can write uniquely as in Lemma \ref{lm_charDD}.
If we define the set 
\begin{equation}\label{eq_defJ}
J=\pg{(m,n,k,l)\in\Z_+^4:\ a_{m,n,k,l}>0}\,,
\end{equation}
then, for a finite set $X=\{(z_1,w_1),(z_2,w_2),\ldots,(z_L,w_L)\}\subseteq\prode$, we can write 
\begin{equation}\label{eq_AxsumBx}
A_X=\sum_{(m,n,k,l)\in J} a_{m,n,k,l} B_X^{m,n,k,l}
\end{equation} where 
\begin{equation}\label{eq_defBX}
B_X^{m,n,k,l}:= [R^{q-2}_{m,n}(z_\mu\cdot z_\nu)\,R^{p-2}_{k,l}( w_\mu\cdot w_\nu)]_{\nu=1,\ldots, L}^{\mu=1,\ldots, L}
\end{equation}
is the positive semidefinite matrix  associated to $X$ and to the  function $R_{m,n}^{q-2}(\xi)R_{k,l}^{p-2}(\eta)$.

With these definitions, the following lemma holds.
\begin{lemma}\label{lm_Ax_sistBx}
The matrix $A_X$  is a positive definite matrix if, and only if, the  equivalence
\begin{equation}\label{eq_sist_iff_B}
	\overline c^t B_X^{m,n,k,l} c = 0\ \ \forall\ (m,n,k,l)\in J\quad  \Longleftrightarrow \quad c=0
\end{equation}
holds true.
\end{lemma}
Lemma \ref{lm_Ax_sistBx} is a consequence of the following one.
\begin{lemma}
 Let $A= \sum_jA_j$, where $A_j$ are positive semidefinite matrices. Then A is positive semidefinite and the condition that $A$ is positive definite is equivalent to 
 $$\overline c^tA_jc= 0 \ \ \forall j \quad  \Longleftrightarrow \quad c=0\,. $$
\end{lemma}
\begin{proof}
First,  $\overline c^tAc=\sum_j\overline c^tA_jc\geq 0$, then one has that $A$ is  positive semidefinite too.
%
%
\\If A is positive definite and  $\overline c^tA_jc=0$ for every $j$, then of course  $\overline c^tAc=0$ and so $c=0$.
\\Finally, if  $\overline c^tAc=0$ then (sum of nonnegative terms) $\overline c^tA_jc=0$ $\forall j$; if we assume that this system implies $c=0$ then $A$ is positive definite.
\end{proof}

In the following proposition we prove one of the two implications of Theorem \ref{th_main}.
\begin{prop}\label{th_spd->progr}
Let $q,p\in\N\cup\pg{\infty}$, $f$ be a  continuous function  which is  PD   on $\prode$ and consider
\begin{equation}\label{eq_defJ'}
J'=\pg{(m-n,k-l)\in\Z^2:\ (m,n,k,l)\in J}\,.
\end{equation}
 If $f$  is  SPD  on $\prode$ then 
\begin{equation}\label{eq_inters}
J'\cap (N\Z+x)\times (M\Z+y)\neq \emptyset \ \text{ for every $N,M,x,y\in\N\,.$}
\end{equation}
\end{prop}
\begin{proof}
Assume $J'\cap (N\Z+x)\times (M\Z+y)= \emptyset$ for some $N,M,x,y\in\N$. Without loss of generality we may assume $M,N\geq2$.
\\
Fix a point  $(z,w)\in \esfq\times\esfp$ and take the set of points  $$X=\pg{(e^{i2\pi \tau/N}z,e^{i2\pi \sigma/M}w)\in\prode:\ \tau=1,..,N,\,\sigma=1,..,M}\,;$$
then, using the Equations  (\ref{eq_modulo-Rmn}-\ref{eq_propridd_Rmn}), the matrix in \pref{eq_defBX} reads as 
$$B_X^{m,n,k,l}=\pq{e^{i2\pi (m-n)(\tau-\lambda)/N}e^{i2\pi (k-l)(\sigma-\zeta) /M}}^{\tau=1,..,N,\,\,\sigma=1,..,M}_{\lambda=1,..,N,\,\,\zeta=1,..,M}. $$ 
Observe that this matrix factors as the product $B_X^{m,n,k,l}=\overline b^tb$ where b is the line vector $$b=\pq{e^{i2\pi (m-n)\tau/N}e^{i2\pi (k-l)\sigma /M}}^{\tau,\sigma}$$
(we omit the dependence on $X$ and $\pt{m,n,k,l}$ in the notation for $b$).
Then each equation of the system in \pref{eq_sist_iff_B} reads as $\overline c^tB_X^{m,n,k,l}c=\overline{ c}^t\overline b^tbc=0$ and is equivalent to  $bc=0$.

At this point we take $c=\pq{e^{-i2\pi \tau x/N}e^{-i2\pi \sigma y /M}}_{\tau,\sigma}$, so that 
\begin{equation}\label{eq_bcsum}
bc=\sum_{\tau,\sigma} e^{i2\pi (m-n-x)\tau/N}e^{i2\pi (k-l-y)\sigma /M}=\sum_{\tau} e^{i2\pi (m-n-x)\tau/N}\sum_{\sigma}e^{i2\pi (k-l-y)\sigma /M}\,.
\end{equation}
By our assumption, for every $\pt{m,n,k,l}\in J$, either  $m-n-x$ is not a multiple of $N$ or $k-l-y$ is not a multiple of  $M$.  This implies that one of the two sums in \pref{eq_bcsum} is zero and then $bc=0$. 
\\Then $c$ is a nontrivial solution of the system in \pref{eq_sist_iff_B}. We have thus proved that  $J'\cap (N\Z+x)\times (M\Z+y)= \emptyset$ implies that $f$ is not SPD.
\end{proof}

The rest of this section is dedicated to proving the following proposition, which contains the remaining implication of Theorem \ref{th_main}.
 \begin{prop}\label{th_progr->spd}
Let $q,p$, $f$ and $J'$ be as in Proposition \tref{th_spd->progr}. If condition \pref{eq_inters} holds true, then $f$ is SPD  on $\prode$.
 \end{prop}
 
First of all, we prove the following consequence of condition \pref{eq_inters}. 
\begin{lemma}\label{lm_int_inf_2}
 	 		If $A\subset\Z^2$ satisfies \begin{equation}\label{eq_inters_lm}
 	 		I_{M,N,x,y}:= A\cap (N\Z+x)\times (M\Z+y)\neq \emptyset \ \text{ for every $N,M,x,y\in\N\,,$}
 	 		\end{equation} then, for every $N,M,x,y\in\N$, the set   $$\pg{\min\pg{|\alpha|,|\beta|}: (\alpha,\beta)\in I_{M,N,x,y}}$$ is unbounded and   $I_{M,N,x,y}$ is infinite.
\end{lemma}
 	\begin{proof}
 		Suppose  $\pg{\min\pg{|\alpha|,|\beta|}: (\alpha,\beta)\in I_{M,N,x,y}}\subseteq [0,C]$.\\
 		Let $(\widehat x,\widehat y) \in(N\Z+x)\times (M\Z+y)$ with $\widehat x,\widehat y>C$ and  $D$ be a multiple of $M$ and of $N$ such that $\widehat x-D,\widehat y-D<-C$. Then $(D\Z+\widehat x)\times (D\Z+\widehat y)\cap I_{M,N,x,y}=\emptyset$ and $$(D\Z+\widehat x) \times (D\Z+\widehat y)\subseteq(N\Z+x)\times (M\Z+y) .$$ As a consequence  $(D\Z+\widehat x)\times (D\Z+\widehat y)\cap A=\emptyset$, which  contradicts \pref{eq_inters_lm}.
 		 	\end{proof}
 
The next step will be to prove that we can verify Strict Positive Definiteness only on certain special sets $X\subseteq\prode$ (see Lemma \ref{lm_SDP_Xenh}).

In view of Lemma \ref{lm_Rto0}, when calculating $\Rmnq(z_\mu\cdot z_\nu)$ and considering the limit for $m+n\to \infty$, the obtained behavior is quite different if $|z_\mu\cdot z_\nu|<1$ or $|z_\mu\cdot z_\nu|=1$. In particular, we will have to treat carefully the cases when $|z_\mu\cdot z_\nu|=1$. 
This happens either if $z_\mu=z_\nu$ (observe that the points in the set $X$ must be distinct but they can have one of the two components in common), or if $z_\mu= e^{i\theta}z_\nu$ with $\theta\in(0,2\pi)$. In this last case we say that the two points $z_\mu,z_\nu\in \esfq$ are {\em antipodal}.
Our strategy to deal with antipodal points is inspired by  \cite{meneg-jean-traira}. We will say that a set  of (distinct) points $Y=\pg{(z_\mu,w_\mu): \mu=1,\ldots, L}$ in $\prode$ is  {\em Antipodal Free} if the following property holds:
\begin{itemize}
\item[(AF)]\quad  if $\mu\neq\nu$ then $|z_\mu\cdot z_\nu|<1$ unless $z_\mu=z_\nu$  and $|w_\mu\cdot w_\nu|<1$ unless $w_\mu=w_\nu$.
\end{itemize}
Of course, since the points in $Y$ are distinct,  if  $z_\mu=z_\nu$ then $|w_\mu\cdot w_\nu|<1$  (resp. if  $w_\mu=w_\nu$ then $|z_\mu\cdot z_\nu|<1$).

\begin{remark}\label{rm_antip1}
Since two distinct points in  $\Omega_2$ are always antipodal, if, for instance, $q=1$, then, in an antipodal free set $Y$ in $\Omega_2\times\Omega_{2p}$,  all the $z_\mu$ are the same and then  $|w_\mu\cdot w_\nu|<1$ for $\mu\neq\nu$. When $p=q=1$ then an antipodal free set $Y$ in $\Omega_2\times\Omega_2$ contains a unique point $(z,w)$. 
\end{remark}
 Consider now an antipodal free set $Y\subseteq \prode$ and  two sets of angles $\Theta=\pg{\theta_\tau: \tau=1,\ldots,t}$ and $\Delta=\pg{\delta_\sigma: \sigma=1,\ldots,s}$ in $[0,2\pi)$.
 We define the {\em enhanced set  associated to} $Y,\,\Theta$ and $\Delta$ as the set 
 \begin{equation}\label{eq_defX}
X=\pg{(e^{i\theta\tau}z_\mu,e^{i\delta_\sigma}w_\mu):\, \mu=1,\ldots, L,\,\tau=1,\ldots,t,\,\sigma=1,\ldots,s}\,.
 \end{equation}
 
Observe that, by construction, the points that appear in $X$ are all distinct (but now there exist many antipodal points among them).

 The following lemma provides a sort of inverse construction.
 \begin{lemma}\label{lm_SsubEnh}
  Given a finite set $S\subseteq \prode$ one can always obtain an antipodal free set $Y\subseteq \prode$  and 
   two sets  $\Theta$ and $\Delta$ of angles in $[0,2\pi)$, such that $S$ is contained in the enhanced set $X$ associated to  $Y,\,\Theta$ and $\Delta$.
 \end{lemma}
  \begin{proof}
  For a finite set $X_1\subseteq \esfq$ one can select a maximal subset $Y_1$ not containing antipodal points and then define the set $\Theta$  containing  $0$  and all the distinct $\theta\in (0,2\pi)$ that are needed to produce the remaining points as $e^{i\theta}z_\mu$ with $z_\mu\in Y_1$.
   
   For the set $S\subseteq \prode$ one produces with this algorithm a  maximal subset $Y_1$ not containing antipodal points along with a corresponding set  of angles $\Theta$ from all the first coordinates $z$ in $S$, than a  maximal subset $Y_2$ not containing antipodal points along with a corresponding set of angles $\Delta$ from all the second coordinates $w$ in $S$.
   
   Then $Y:=Y_1\times Y_2$ will be such that $S$   is contained in the enhanced set  associated to $Y,\,\Theta$ and $\Delta$.
   \end{proof}

The following two lemmas will make clear why it is useful to consider antipodal free sets.
 \begin{lemma}\label{lm_PDlimit}
  Let $Y=\pg{(z_\mu,w_\mu): \mu=1,\ldots, L}$ in $\prode$  be antipodal free. Then the matrix 
  $$ [R_{m,n}^{q-2}(z_\mu\cdot z_\nu)R_{k,l}^{p-2}(w_\mu\cdot w_\nu)]_\nu^\mu \quad $$ is positive definite provided $n\neq m,\,k\neq l$ and $m+n,k+l$ are large enough.
 \end{lemma}
 \begin{proof}
  Actually, 	the diagonal elements of the matrix are all equal to $R_{m,n}^{q-2}(1)R_{k,l}^{p-2}(1)=1$, moreover,
 condition (AF) implies that  if  $z_\mu\cdot z_\nu=1$ then $|w_\mu\cdot w_\nu|<1$ and  if  $w_\mu\cdot w_\nu=1$ then $|z_\mu\cdot z_\nu|<1$.
  As a consequence,  the non-diagonal elements converge to zero by \pref{eq_Rto0}, when $n\neq m,\,k\neq l$ and $min\pg{m+n,k+l}\to \infty$. Then  the matrix, which is  Hermitian and with real positive diagonal, becomes strictly diagonally dominant, thus positive definite (\cite[Theorem 6.1.10]{horn-joh-matrix}).
  \end{proof}

     \begin{lemma}\label{lm_SDP_Xenh}
      Let $q,p\in\N\cup\pg{\infty}$ and $f$ be a  continuous function  which is  PD   on $\prode$.  Then the  following assertions are equivalent:
     \begin{itemize}
     \item[(i)] $f$ is SPD on $\prode$;
     
         \item[(ii)] the matrix $A_X$ defined in \pref{eq-def-AX} is positive definite for every finite set $X$ being the enhanced set associated to some antipodal free set $Y\subseteq \prode$ and  two sets  $\Theta$ and $\Delta$ of angles in $[0,2\pi)$.
         \end{itemize}
      \end{lemma}
  \begin{proof} 
  First observe that  $(i)$ is equivalent to:
  \begin{equation*}
 \text{(iii) $A_S$ is a positive definite matrix for every finite set  $S\subseteq \prode$.}
  \end{equation*}
  The implication  $(iii)\Longrightarrow(ii)$ is trivial. In order to prove that  $(ii)\Longrightarrow(iii)$ observe that, given $S$, one can obtain $X$ as described in Lemma \ref{lm_SsubEnh}: since $S\subseteq X$, then  $A_S$ is a principal submatrix of the positive definite matrix $A_X$ and then it is a positive definite matrix itself. 
  \end{proof}

At this point we can prove Proposition \ref{th_progr->spd}. 
 \begin{proof}[Proof of Proposition \ref{th_progr->spd}]
Let  $X$ (finite) be the enhanced set associated to an  antipodal free set $Y\subseteq \prode$  and two sets  $\Theta$ and $\Delta$ of angles in $[0,2\pi)$ and  
 consider the system 
 \begin{equation}\label{eq_sistB}
  \overline c^tB_X^{m,n,k,l}c=0\ \text{for every $(m,n,k,l)\in J$}.
  \end{equation}
  In view of the Lemmas \ref{lm_Ax_sistBx} and \ref{lm_SDP_Xenh}, all we have to do is to prove that this system implies $c=0$.

 Using the property in \pref{eq_propridd_Rmn}, with the notation introduced in \pref{eq_defX} for the elements of $X$,  we have
  $$B_X^{m,n,k,l}=\pq{e^{i (m-n)(\theta_\tau-\theta_\lambda)}e^{i(k-l)(\delta_\sigma-\delta_\zeta) }R_{m,n}^{q-2}(z_\mu\cdot z_\nu)R_{k,l}^{p-2}(w_\mu\cdot w_\nu)}^{\tau,\sigma,\mu}_{\lambda,\zeta,\nu}\,.$$
It is convenient to write this matrix as a block matrix as follows: 
 
$$ B_X^{m,n,k,l}=[R_{m,n}^{q-2}(z_\mu\cdot z_\nu)R_{k,l}^{p-2}(w_\mu\cdot w_\nu)A^{m,n,k,l}]_\nu^\mu$$  where $$A^{m,n,k,l}=[e^{i (m-n)(\theta_\tau-\theta_\lambda)}e^{i(k-l)(\delta_\sigma-\delta_\zeta) }]^{\tau,\sigma}_{\lambda,\zeta}\,.$$
 The vector $c$ will be correspondingly split as
$$c=[c_\mu]_\mu\quad\text{  where }\quad c_\mu=[c_\mu^{\tau\sigma}]_{\tau,\sigma}\,.$$
We have then $$\overline c^tB_X^{m,n,k,l}c=\sum_{\mu,\nu} R_{m,n}^{q-2}(z_\mu\cdot z_\nu)R_{k,l}^{p-2}(w_\mu\cdot w_\nu)\overline{c_\nu}^tA^{m,n,k,l}c_\mu.$$

 Similar to the proof of Proposition \ref{th_spd->progr}, the matrix $A^{m,n,k,l}$ factors as $A^{m,n,k,l}=\overline{b}^tb$ where $$b=[{e^{i (m-n)\theta_\tau}e^{i (k-l)\delta_\sigma}}]^{\tau\sigma}\,,$$
 then we may write 
 \begin{equation}\label{eq_cBc}
 \overline c^tB_X^{m,n,k,l}c=\sum_{\mu,\nu}  \overline{bc_\nu}^tbc_\mu R_{m,n}^{q-2}(z_\mu\cdot z_\nu)R_{k,l}^{p-2}(w_\mu\cdot w_\nu)\,.
 \end{equation} 
 
 Observe that since $Y$ is antipodal free we will be able to use Lemma \ref{lm_PDlimit} in order to discuss this quadratic form.

 We  suppose now for the sake of contradiction that $c\neq0$. Without loss of generality we assume that $c_1^{1,1}\neq0$ and we first aim to prove that 
 \begin{equation}\label{eq_bcneq0} 
 bc_1=\sum_{\tau,\sigma}{e^{i (m-n)\theta_\tau}e^{i (k-l)\delta_\sigma}}c_1^{\tau,\sigma}\neq0
 \end{equation} 
 for certain $(m,n,k,l)\in J$.
 
  Actually, by the Theorem 2.4 and the Lemmas 2.5 and 2.6 in \cite{P-jean-menS1xS1}, which use the theory of linear recurrence sequences, and in particular a generalization of the Skolen-Mahler-Lech  Theorem due to Laurent \cite[Theorem 1]{Laurent89}  (see also \cite{pinkus-spd-herm}),
  we know that  given the angles $\theta_\tau, \delta_\sigma$ and the vector $c_1$, with  $c_1^{1,1}\neq0$, there exist $N,M,x,y\in\N$ such that the function defined in $\Z^2$
  $$  L( \alpha,\beta):=\sum_{\tau,\sigma}{e^{i\, \alpha\,\theta_\tau}e^{i\, \beta\;\delta_\sigma}}c_1^{\tau,\sigma}$$
 is not zero 
 for all $( \alpha,\beta)$ in the set   $P:=(N\Z+x)\times (M\Z+y)$.

  By Lemma \ref{lm_int_inf_2}  applied to $J'$,   there exists a sequence $S:=\pg{(\alpha_i,\beta_i)}\subseteq P\cap J'$ such that $|\alpha_i|,|\beta_i|\to\infty$.
  As a consequence,  \pref{eq_bcneq0} holds true for every  $(m,n,k,l)\in J$ such that $(m-n,k-l)\in S$.
\\	Now we can select $(m-n,k-l)\in S$ with $|m-n|,|k-l|$ as large as we want (which implies that $ m\neq n$, $k\neq l$ and that $m+n$ and $k+l$ are also large).  
For the corresponding  $(m,n,k,l)\in J$, the equation in \pref{eq_sistB} can not be zero in view of Equation \pref{eq_bcneq0} and Lemma \ref{lm_PDlimit}.

 We have then proved that  a nontrivial solution of system \pref{eq_sistB} can not exist.
\end{proof}
\begin{remark}
Observe that in the case $p=q=1$, in view of Remark \ref{rm_antip1}, the sum in Equation \pref{eq_cBc} has only one term which is  $|bc_1|^2$, then the contradiction follows readily after proving \pref{eq_bcneq0}.
\end{remark}

At this point, Theorem \ref{th_main} is a consequence of the Propositions \ref{th_spd->progr} and  \ref{th_progr->spd}. The Theorems \ref{th_main_1p} and \ref{th_main_11} follow from the same two propositions after 
translating back from the expansion in Lemma \ref{lm_charDD} to the usual ones in the Equations \pref{eq-pd-prod-esf-complO2p} and \pref{eq-pd-prod-esf-complO2} (see in the Sections \ref{sec_PD_on_sing} and \ref{sec_PD_on_prod}).

 \section{Characterization of Positive Definiteness on $\esfi\times\esfi$}\label{sec_infty}  
 In this section we aim to prove the following:
 \begin{theorem}\label{th_PDinfty} Let $f:\D\times \D\to\C$ be a continuous function. Then $f$ is PD on $\esfi\times\esfi$ if, and only if,
	\begin{equation}\label{eq:expandcpinfty}
 \begin{array}{c}
	\begin{array}{rcll}
f(\xi,\eta)&=&\displaystyle\sum_{m,n,k,l\in\Z_+} a_{m,n,k,l}R_{m,n}^\infty(\xi)R_{k,l}^\infty(\eta)&\\
 	&=& \displaystyle\sum_{m,n,k,l\in\Z_+} a_{m,n,k,l} {\xi\vphantom{\overline\xi}}^m\overline{\xi}^n\eta^k\overline{\eta}^l,\quad& (\xi,\eta)\in{\D}\times{\D},\end{array}
\\\mbox{where $\sum a_{m,n,k,l}<\infty$ and $a_{m,n,k,l}\geq0$ for all $m,n,k,l\in\Z_+$.}
\end{array}
 	\end{equation}
 	 	Moreover, the series in Equation \pref{eq:expandcpinfty} is uniformly convergent on  ${\D}\times {\D}$.
 \end{theorem}
 
 In the proof we will use ideas from \cite{P-berg-porcu-Omega-inf} 
 and we will need the following lemma, whose proof is analogous to that of Lemma 4.1 in \cite{P-berg-porcu-Omega-inf} and will be omitted. 
 \begin{lemma}\label{thm:technical}
 Let $q,p\in\N\cup\pg\infty,\;q,p\geq2$ and  $f:\D\times \D\to \C$ be a continuous and PD function on $\prode$. Given   points $w_1,\ldots,w_L\in\esfp$ and numbers $c_1,\ldots,c_L\in \C$, the function  $F:\D\to \C$ defined by
 \begin{equation}\label{eq:sum1}
 F(\xi)=\sum_{j,k=1}^L f(\xi,w_j\cdot w_k)c_j\overline{c_k}
 \end{equation} is continuous and PD on $\esfq$. 
 \end{lemma}
  \begin{proof}[Proof of Theorem \ref{th_PDinfty}]
  First observe that $f$  is PD on $\esfi\times\esfi$ if, and only if, $f$  is PD on $\prode$ for every $q,p\geq2$.

  It is also easy to see that the function  $g(\xi)=\xi$, $\xi\in{\D}$, 
  is PD on $\Omega_{2q}$ for every $q\geq2$, as well as its conjugate. 
  By the Schur Product Theorem for Positive Definite kernels, cf. \cite[Theorem 3.1.12]{Berg}, one obtains  that also $h(\xi)={\xi\vphantom{\overline\xi}}^{m}\overline{\xi}^{n}$ is PD on $\Omega_{2q}$ for $q\geq2$ and $m,n\in\Z_+$, and that ${\xi\vphantom{\overline\xi}}^m\overline{\xi}^n\eta^k\overline{\eta}^l$ is  PD on $\Omega_{2q}\times\Omega_{2p}$ for $q,p\ge 2$  and $m,n,k,l\in\Z_+$.
  As a consequence,  any function of the form \pref{eq:expandcpinfty} is continuous and PD on $\Omega_{2q}\times\Omega_{2p}$ for every $q,p\geq2$, and then on $\Omega_{\infty}\times\esfi$ too. 
 
 \par \medskip

  Now let the continuous function $f:{\D}\times {\D}\to\C$ be PD on $\esfi\times\esfi$. 
 For $\eta\in\D,\,c\in\C$, consider the special case of \pref{eq:sum1} with $L=2,q=\infty,p=2$, $w_1=(\eta,w), w_2=(1,0)\in\Omega_4$, $c_1=1, c_2=c$, that is,
   \begin{equation}\label{eq:sum2}
   F_{\eta,c}(\xi)=f(\xi,1)(1+|c|^2)+f(\xi,\eta)\overline{c}+f(\xi,\overline{\eta})c.
   \end{equation}
By Lemma~\ref{thm:technical},
   $F_{\eta,c}$ is a continuous  PD function  on $\Omega_\infty$. Then, using a theorem due to Christensen and Ressel, see \cite{chris-ressel-pd}, it can be written as
   $$
   F_{\eta,c}(\xi)=\sum_{m,n\in\Z_+} a_{m,n}\pt{\eta,c} {\xi\vphantom{\overline\xi}}^m\overline{\xi}^n,
   $$
   where $a_{m,n}\pt{\eta,c}\ge 0$ are uniquely determined  and satisfy $\sum_{m,n\in\Z_+} a_{m,n}\pt{\eta,c}<\infty.$
 
 By using  $c=1,-1,i$ and proceeding as in the end of the proof of \cite[Theorem 1.2]{P-berg-porcu-Omega-inf}, one obtains that  
  \begin{equation}\label{eq-cara-f}
   f(\xi,\eta)=\frac{1-i}4F_{\eta,1}(\xi)-\frac{1+i}4F_{\eta,-1}(\xi)+\frac{i}2F_{\eta,i}(\xi)=\sum_{m,n\in\Z_+} \varphi_{m,n}(\eta){\xi\vphantom{\overline\xi}}^m\overline{\xi}^n,
   \end{equation}
   where
   $$
   \varphi_{m,n}(\eta):=\frac{1-i}4a_{m,n}(\eta,1)-\frac{1+i}4a_{m,n}(\eta,-1)+\frac{i}2a_{m,n}(\eta,i), \quad \eta\in\D\,,
   $$
  and then  
   \begin{equation}\label{eq-serie-phi-finita}
   \m{\sum_{m,n\in\Z_+}\varphi_{m,n}(\eta)}<\infty,\qquad \eta\in{\D}.  
   \end{equation}
 
  Consider now $p\geq2$ and  the function $\widetilde  f_p:\D\times U(p):(\xi,A)\mapsto f(\xi,A e_p\cdot e_p)$, where $e_p=(1,0,\ldots,0)\in \esfp$.
By construction, $\widetilde f_p$ is continuous and  PD on $\Omega_{\infty}\times  U(p)$. 
By Theorem 1.3 in \cite{P-berg-porcu-Omega-inf}, we can expand $\widetilde f_p$ as
$$\widetilde f_p(\xi, A )=\sum_{m,n\in\Z_+} \widetilde\varphi^{(p)}_{m,n}(A)R^{\infty}_{m,n}(\xi)=\sum_{m,n\in\Z_+} \widetilde\varphi^{(p)}_{m,n}(A){\xi\vphantom{\overline\xi}}^m\overline{\xi}^n\,,$$
where $\widetilde\varphi^{(p)}_{m,n}$ are continuous PD functions on $U(p)$.

By derivation one has that
$$ \widetilde  \varphi^{(p)}_{m,n}(A) =\frac1{m!n!}\frac{\partial^{m+n}\widetilde f_p(0,A)}{\partial {\xi\vphantom{\overline\xi}}^m\partial\overline{\xi}^n}$$
and
  \begin{equation}\label{eq_phi=der}
     \varphi_{m,n}(\eta) =\frac1{m!n!}\frac{\partial^{m+n}f(0,\eta)}{\partial {\xi\vphantom{\overline\xi}}^m\partial\overline{\xi}^n},
     \end{equation}
but by construction 
 $$\widetilde  \varphi^{(p)}_{m,n}(A)= \frac1{m!n!}\frac{\partial^{m+n}\widetilde f_p(0,A)}{\partial {\xi\vphantom{\overline\xi}}^m\partial\overline{\xi}^n}=\frac1{m!n!}\frac{\partial^{m+n}f(0,Ae_p\cdot e_p)}{\partial {\xi\vphantom{\overline\xi}}^m\partial\overline{\xi}^n}= \varphi_{m,n}(Ae_p\cdot e_p).
    $$
 By Remark \ref{rem_U_Om} we deduce that  $\varphi_{m,n}$ is continuous and  PD on $\Omega_{2p}$, for every $p\geq2$. 
  As a consequence, $\varphi_{m,n}$ is PD on $\esfi$ and thus  we can again use the  theorem by   Christensen and Ressel, in order to conclude that for every  $m,n$, 
          $$
            \varphi_{m,n}(\eta) = \sum_{k,l\in\Z_+} a_{m,n,k,l}\,\eta^k\overline{\eta}^l, \quad \eta\in{\D}\,,
            $$ where ${a_{m,n,k,l}}\geq0$, for every $k,l\in\Z_+$, and 
        $    \sum_{k,l\in\Z_+} a_{m,n,k,l}<\infty\,.  $
        Thus,
        $$
        f(\xi,\eta)=\sum_{m,n\in\Z_+}\sum_{k,l\in\Z_+} a_{m,n,k,l}\,{\xi\vphantom{\overline\xi}}^m\overline{\xi}^n\eta^k\overline{\eta}^l\,,
        $$
        and then $\displaystyle\sum_{m,n,k,l\in\Z_+}a_{m,n,k,l}<\infty.$
     \end{proof}

 \section{A connection with the cases  $S^1$ and $S^1\times S^1$}
\label{sec_S1}
 In this section we aim to show
 that  one can deduce, from Theorem \ref{th_main_11},  the characterization of Strict Positive Definiteness on $S^1\times S^1$ proved in \cite{P-jean-menS1xS1},
 namely, that 
  a continuous function $f:[-1,1] \times [-1,1] \to \C$    
   which is  PD   on $S^1\times S^1$,  is also SPD  on $S^1\times S^1$ if, and only if, considering its expansion as in \pref{eq-pd-esfSd},
 the set 
 $
  \{(m,k)\in\Z^2: a_{|m|,|k|}>0\}
  $
  intersects every product of full arithmetic progressions in $\Z$, that is, 
  \begin{equation}\label{eq_inters_S11}
	\{(m,k)\in\Z^2: a_{|m|,|k|}>0\}\cap (N\Z+x)\times (M\Z+y)\neq \emptyset\qquad \mbox{for every $N,M,x,y\in\N$.}
	\end{equation}  

Actually, condition \pref{eq_inters_S11} has more similarities with the conditions we obtain here in the Theorems \ref{th_main}, \ref{th_main_1p} and \ref{th_main_11} for the complex spheres, where an intersection with  every product of full arithmetic progressions in $\Z$ is required, rather than with the known conditions for real spheres in higher dimensions, where only progressions of step 2 are involved (see Equations \pref{eq_inters_Sd} and \pref{eq_inters_SpSq}). 

The polynomials  $P_m^0$ in \pref{eq-pd-esfSd} are also known as  Tchebichef polynomials of the first kind (see \cite[page 29]{szego}) and can be written as  $P_m^{0}(\cos\phi)=\cos(m\phi)$, $\phi\in[0,\pi]$. As a consequence, a way of writing  \pref{eq-pd-esfSd} often used in literature  when $p=q=1$ is the following:
   \begin{equation}\label{eq_expS1S1}
    f(\cos(\phi),\cos(\psi))=\sum_{m,k\in\Z_+} a_{m,k}\cos(m\phi)\cos(k\psi),\quad \phi,\psi\in[0,\pi].
    \end{equation}

Below, we will show that one can establish a correspondence between PD (and between SPD) functions on  $S^1\times S^1$ and a subset of those on $\Omega_2\times\Omega_2$. We will do it first for the case of a single sphere.

\begin{lemma}\label{lm_bij_1}
There exists a bijection between PD (resp. SPD) functions on $S^1$ and PD (resp. SPD) functions on $\Omega_2$ which are invariant under conjugation, that is, $f(e^{i\phi})=f(e^{-i\phi}),\ \phi\in[0,2\pi)$.
\end{lemma}
\begin{proof}
Let $f:\partial\D\to \C$ be a PD function on $\Omega_2$  satisfying  $f(e^{i\phi})=f(e^{-i\phi})$, then it is real valued and it only depends on the real part. 
\\
Consider the bijection
$$A:\Omega_2\to S^1 :e^{i\phi}\mapsto (\cos(\phi),\sin(\phi))$$
and the surjective map

$$C:\partial\D\to[-1,1]:e^{i\phi}\mapsto \cos(\phi)\,,$$
 which admits a right inverse $C^-:x\mapsto e^{i\arccos(x)}$. 
Then $C\circ C^-=id_{[-1,1]}$ and	 since $f$ only depends on the real part,
	\begin{equation}\label{eq-fCC}
	f(C^-\circ C(e^{i\phi}))=f(e^{i\phi}), \quad e^{i\phi}\in\partial \D.
	\end{equation}	
Also observe that
	\begin{equation}\label{prod-int-w-compl}
	C(w\cdot w')=Aw\cdot_\R Aw', \quad w,w'\in\Omega_2.
	\end{equation}

Therefore, the bijection in the claim is the following:
$$B:f\mapsto \widehat f:=f\circ {C^-}\,,$$
whose inverse is given by $$B^{-1}: \widehat f\mapsto f:=\widehat f \circ C\,.$$
Actually, for kernels $K$ and $\widehat K$ associated, respectively, to $f$ and $\widehat f$,  it holds, by (\ref{eq-fCC}-\ref{prod-int-w-compl}),
$$\widehat K(Aw,Aw')=\widehat f(Aw\cdot_\R Aw')=f(C^-\circ C (w\cdot w'))=f(w\cdot w')=K(w,w'),$$
then  the definition of  PD (resp. SPD) in \pref{eq-quad-form-geral} becomes equivalent for the two kernels.
\end{proof}

The case on a product of spheres is very similar.
\begin{lemma}\label{lm_bij_11}
	There exists a bijection between PD (resp. SPD) functions on $S^1\times S^1$ and PD (resp. SPD) functions on $\Omega_2\times\Omega_2$ that 
	are invariant under conjugation in both variables, that is, $f(e^{i\phi},e^{i\psi})=f(e^{-i\phi},e^{i\psi})=f(e^{i\phi},e^{-i\psi}),\ \phi,\psi\in[0,2\pi)$.	
\end{lemma}
\begin{proof}
Consider  a  function $f:\partial\D\times\partial\D \to \C$ which is PD on $\Omega_2\times\Omega_2$ and satisfies  $f(e^{i\phi},e^{i\psi})=f(e^{-i\phi},e^{i\psi})=f(e^{i\phi},e^{-i\psi})$, then it is real valued and it only depends on the real part of both $e^{i\phi}$ and $e^{i\psi}$. 

Thus the proof follows the same lines as that of Lemma \ref{lm_bij_1}, where now  the bijection is defined as 
$$B:f\mapsto \widehat f(\vartriangle,\star):=f( {C^-(\vartriangle),C^-(\star))}\,.$$
\end{proof}

Now we need to establish the correspondence between the coefficients in the expansions of $f$ and $\widehat f$.

For the single sphere case,
a  continuous  function  $f$ which is  PD   on $\Omega_2$ can be written as in \pref{eq-pd-prod1}.
	The condition that $f$ is   invariant under conjugation, assumed in Lemma \ref{lm_bij_1},
 	is equivalent to  $a_m=a_{-m},\ m\in\Z$, then we can rewrite  $$f(e^{i\phi})=a_0+\sum_{m\in\N} {a_{m}} (e^{im\phi}+e^{-im\phi})= a_0+\sum_{m\in\N} {2a_{m}} \cos(m\phi)$$
 	 and the  function $\widehat f$ corresponding to $f$ in the bijection from Lemma \ref{lm_bij_1} can be written as 
	$$\widehat f(\cos(\phi))
=f(C^-(\cos(\phi)))=f(e^{i\phi}) 
 	 	= a_0+\sum_{m\in\N} {2a_{m}} \cos(m\phi)\,.$$
If we consider the Schoenberg coefficients $\widehat a_m,\,m\in\Z_+$, for $\widehat f$, that is,  
$$ \widehat f(\cos(\phi))=\widehat a_0+\sum_{m\in\N} {\widehat a_{m}} (\cos(m\phi))\,,$$
we obtain the relation  
\begin{equation}\label{eq_relcoefS1O2}
a_{m}>0\Longleftrightarrow \widehat a_{|m|}>0,\ \ m\in\Z\,.
\end{equation}
Now, by \cite{P-valdir-claudemir-spd-compl},  the condition for $f$ to be  SPD on $\Omega_2$ is
 	 \begin{equation}\label{eq_intersO2}
\pg{m\in\Z:\ a_{m}>0} \cap (N\Z+x)\neq \emptyset\qquad \mbox{for every $N,x\in\N$,}
 	 	 \end{equation}
which then translates, via \pref{eq_relcoefS1O2}, to the known condition (see  \cite{P-valdir-claudemir-spd-compl,barbosa-men})
\begin{equation}
 	 \pg{m\in\Z:\ a_{|m|}>0} \cap (N\Z+x)\neq \emptyset\qquad \mbox{for every $N,x\in\N$.}
 	 	 \end{equation}

Again, when considering the product of two spheres, the argument is similar.
 	 A continuous function $f$, which is  PD  on $\Omega_2\times\Omega_2$, is written as in \pref{eq-pd-prod-esf-complO2}
 	 	and 	the condition that $f$ is invariant  under conjugation in both variables, assumed in Lemma \ref{lm_bij_11},
 	 	 	 	is equivalent to 
	 $$a_{m,k}=a_{-m,k}=a_{m,-k}=a_{-m,-k},\ m,k\in\Z.$$ 
Then,  proceeding as above, one obtains 
 	 \begin{eqnarray}
\nonumber f(e^{i\phi},e^{i\psi})&=& a_0+\sum_{m\in\N} {2a_{m,0}} \cos(m\phi)+\sum_{k\in\N} {2a_{0,k}} \cos(k\psi)+\sum_{m,k\in\N} {4a_{m,k}} \cos(m\phi)\cos(k\psi)\\
 	 	&=&\widehat f(\cos\phi,\cos\psi)\,.
\end{eqnarray}	 	 	 
If we denote by  $\widehat a_{m,k},\,m,k\in\Z_+$, the coefficients in the expansion of  $\widehat f$ as in \pref{eq_expS1S1}, 
we obtain the relation 
\begin{equation}\label{eq_relcoefS1O2_quad}
a_{m,k}>0\Longleftrightarrow \widehat a_{|m|,|k|}>0\ \ m,k\in\Z\,.
\end{equation}

Finally,  by Theorem \ref{th_main_11},  the condition for $f$ to be SPD on $\Omega_2\times\Omega_2$ is 
 	 \begin{equation}\label{eq_intersO2O2}
 	\pg{(m,k)\in\Z^2:\ a_{m,k}>0} \cap  (N\Z+x)\times (M\Z+y)\neq \emptyset\qquad \mbox{for every $N,M,x,y\in\N$,}
 	 	 	 \end{equation}
 which then translates, via \pref{eq_relcoefS1O2_quad}, to the condition \pref{eq_inters_S11} that we were seeking.

	\section*{Acknowledgement}
	Mario H. Castro was supported by: grant $\#$APQ-00474-14, FAPEMIG and CNPq/Brazil.
	\\
		Eugenio Massa was  supported by: grant $\#$2014/25398-0, São Paulo Research Foundation (FAPESP) and  grant $\#$308354/2014-1, CNPq/Brazil. 
	\\
	 Ana P. Peron was supported by: grants $\#$2016/03015-7, $\#$2016/09906-0 and $\#$2014/25796-5, S\~ao Paulo Research Foundation (FAPESP).

\providecommand{\bysame}{\leavevmode\hbox to3em{\hrulefill}\thinspace}

\end{document}